\documentclass[12pt]{amsart}
\usepackage{amsmath,amssymb, amsfonts,textcomp, amsthm, mathtools}
\usepackage{subfigure}

\usepackage{lmodern}
\usepackage[T1]{fontenc}
\usepackage{tikz}
\usepackage{tikz-cd}
\usepackage{breqn}
\usepackage{comment}
\usepackage{hyperref}
\usepackage[utf8]{inputenc}

\theoremstyle{plain}
\newtheorem{theorem}{Theorem}[section]
\newtheorem{lemma}[theorem]{Lemma}
\newtheorem{corollary}[theorem]{Corollary}
\newtheorem{proposition}[theorem]{Proposition}

\theoremstyle{definition}
\newtheorem{definition}[theorem]{Definition}
\newtheorem{example}[theorem]{Example}

\theoremstyle{remark}

\usepackage{color}
\usepackage{fancyvrb}

\DefineVerbatimEnvironment{Highlighting}{Verbatim}{commandchars=\\\{\}}
\usepackage{framed}
\definecolor{shadecolor}{RGB}{248,248,248}

\tikzset{cong/.style={draw=none,edge node={node [sloped, allow upside down, auto=false]{\(\cong\)}}},
	Isom/.style={draw=none,every to/.append style={edge node={node [sloped, allow upside down, auto=false]{\(\cong\)}}}}}
\tikzset{sim/.style={draw=none,edge node={node [sloped, allow upside down, auto=false]{\(\sim\)}}},
	Sim/.style={draw=none,every to/.append style={edge node={node [sloped, allow upside down, auto=false]{\(\sim\)}}}}}

\begin{document}
	\title{A Martindale-like Ring of Quotients for any ring}
	\author{K.S. Enoch Lee}
	\address{Department of Mathematics , Auburn University at Montgomery, Montgomery, AL 36117}
	\email{elee4@aum.edu}
		\subjclass[2020]{16S85, 16P70, 16S90}
\keywords{Martindale, rings of quotients, dense ideals, centers}

		\begin{abstract}
	In this paper we introduce a Martindale-like ring of quotients for any ring called sub-Martindale  ring of quotients. When the base ring is semiprime, its maximal sub-Martindale ring of quotients is just the well-known Martindale ring of quotients. In a similar fashion we introduce some rings of quotients for any ring similar to Martindale symmetric ring of quotients.
		\end{abstract}
	\maketitle

\section{Introduction}

		Utumi~\cite{Utumi56} introduced the notion of general ring of quotients and showed that rings without total left zero divisors are those  which have  maximal right rings of quotients. These maximal right rings of quotients can be characterized by the study of dense right ideals of their base rings.

		Martindale~\cite{Martindale69} introduced a special ring of quotients for prime rings by investigating ideals that are also dense right ideals. Later Amitsur~\cite{Amitsur72} extended this concept to semiprime rings. These  rings of quotients are called Martindale rings of quotients. We extend the construction to any rings. 
		
		In the preliminaries, containing basic definitions, we introduce the notion of right (or left) Martindale sets. In the Section 2, we associate to each Martindale set a ring of quotients and call it sub-Martindale ring of quotients. We show that some well-known properites of the  Martindale ring of quotients (of a semiprime ring) are also true in  sub-Martindale ring of quotients. We demonstrate the existence of a unique maximal sub-Martindale ring of quotients and show that this is equivalent to the Martindale ring of quotients when the base ring is semiprime.  
 Section 3 is devoted to the development  of sub-Martindale symmetric ring of quotients. It is shown that a unique maximal sub-Martindale symmetric ring of quotients for any ring exists.
We then investigate the notion of center of a ring and the relation to its sub-Martindale rings of quotients in Section 4. Examples are presented in the last section. 
		
		\section{{Preliminaries}}
		
		All rings are (not necessary commutative) associative with a unity.  Suppose
		\(R\) is a ring, we denote \(1_R\) the unity of \(R\).  A ring
		homomorphism preserves unity unless stated otherwise. A subring, \(S\), of a ring \(R\)
		is a ring which shares the same unity with \(R\). In this case, \(R\) is called an
		extension ring  of \(S\).   In general, if
		there is a ring monomorphism from \(S\) to \(R\), we identify \(S\) with its image
		in \(R\) and consider \(R\) as an  extension ring of \(S\). 		
		We denote by
		\(N_R \leqslant_r M_R\) if \(N\) is a right \(R\)-submodule of \(M_R\). Similarly \(_RN \leqslant _l {}_RM\) means \(N\) is a left \(R\)-submodule of \(_RM\).
		
		Given a submodule  \(N_R \leqslant _r M_R\) and letting \(y\in M\), we define
		\[y^{-1}N=\{r\in R : yr\in N \}.\] Note this is a right ideal of \(R\). We call \(N_R\)
		a {\em dense right submodule\/} of \(M_R\) (written  \(N_R\leqslant_r ^{den} M_R\)) if for any
		\(x, y\in M\) with \(x\not = 0\) such that \(x(y^{-1}N)\not = 0\).
		We also write 
		\(I\leqslant_r ^{den} R\) if \(I\) is a dense right ideal of \(R\).
		Similarly \(\leqslant_l ^{den}\) means dense left submodule or dense left ideal.
		An ideal \(I\) of \(R\) is denoted
		by \(I \unlhd R\). Furthermore, \(I \unlhd_r^{den} R\) (resp. \(I \unlhd_l^{den} R\)) indicates the ideal
		\(I\) is also a dense right (resp. left) ideal of \(R\). Given two right (resp. left) ideals  \(I\) and \(J\) of \(R\), we denote \(I+J\) and \(IJ\) the smallest right (resp. left) ideals of \(R\)
		containing \(I\cup J\) and \(\{ij| i\in I, j\in J\}\), respectively.

		We write \(Q^r_{\max}(R)\) and \(Q^l_{\max}(R)\) to denote
		the maximal right and left rings of quotients respectively.
		Similarly \(Q^r_{cl}(R)\) and \(Q^l_{cl}(R)\) are  respectively
		the classical right and left rings of quotients.
		When \(R\) is semiprime, the Martindale right ring of quotients \cite[14.7]{Lam99}  is  \[Q^r(R)=\{f\in Q_{\max}^r(R)\mid fI\subseteq R \text{ for some } I \unlhd_r^{den} R\}.\]  Similarly, the Martindale left ring of quotients is \[Q^l(R)=\{f\in Q_{\max}^l(R)\mid If\subseteq R \text{ for some } I \unlhd_l^{den} R\}.\]
		
		We write \({\bf {r}}_R(X)\) and
		\({\bf l}_R(X)\) for a set \(X\) the right annihilator and left annihilator in \(R\)
		respectively.
		
		We shall use the fact \cite[p.236]{Passman91} that an ideal \(I\) of \(R\)  is also a dense right ideal (resp. a dense left ideal) if and only
		if \({\bf l}_R(I)=0\) (resp. \({\bf r}_R(I)=0\)).

		\begin{definition}
			Given a ring \(R\), let \(T\) be a non-empty collection of some dense right ideals of \(R\) and
			\(F\) be an extension ring  of \(R\). We say \(F\) is a
			{\em \(T\)-extension} of \(R\) if the following conditions are true.
			\begin{description} 
				\item [{(MR.1)}]  If \(f\in F\), then there is an \(I\in T\) such that  \(fI \subseteq R\).
				\item [{(MR.2)}]   If \(f\in F\) and \(I\in T\), then \(fI=0\) implies that \(f=0\).
				\item [{(MR.3)}]  If \(I\in T\) and \(\sigma \in Hom_R(I_R,R_R)\), then \(\exists f\in F\)
				such that \(\sigma (i)=fi\) for all \(i\in I\).
			\end{description}
			A dense left ideals version of the above conditions
			can be defined similarly. Let \(T\) be a non-empty collection of some dense left ideals of \(R\).
			\begin{description}
				\item [{(ML.1)}]  If \(f\in F\), then there is an \(I\in T\) such that  \(If \subseteq R\).
				\item [{(ML.2)}]  If \(f\in F\) and \(I\in T\), then \(If=0\) implies that \(f=0\).
				\item [{(ML.3)}]  If \(I\in T\) and \(\phi \in Hom_R({_R}I,{_R}R)\), then \(\exists f\in F\)
				such that \((i)\phi =if\) for all \(i\in I\).
			\end{description}
		\end{definition}
		
If \(T\) is the set of all dense right ideals of \(R\), then \(Q_{\max}^r(R)\) can be characterized (\cite[2.1.7]{Beidar96},\cite[24.8]{Passman91}) as a \(T\)-extension of \(R\).	Moreover, when \(R\) is a semiprime ring and \(T\) is the set of all ideals that are also dense right ideals of \(R\), the Martindale right ring of quotients  is a \(T\)-extension of \(R\) (see \cite[2.2.1]{Beidar96},  \cite[14.24]{Lam99}, or \cite[1.2]{Passman87}).
		
		\begin{definition}\label{MartindaleSet}
			Let \(T\) be a non-empty subset of ideals of \(R\) that are also  dense right ideals. We call
			\(T\) a {\em right Martindale set of \(R\)}
			\begin{enumerate}
				\item[(1)]    if \(A\in T\) and \(B\) is an ideal of \(R\) containing \(A\), then  \(B\in T\); and
				\item[(2)]  if \(A,B\in T\), then \(AB\in T\).
			\end{enumerate}
			Note that \(R\) is necessarily in \(T\).
			Recall that the product of two dense right ideals is also dense. Thus the collection of
			all ideals that are also dense right ideals is a right Martindale set of \(R\). In fact this is the unique largest right Martindale set of \(R\).
			
			A {\em left Martindale set\/} can be defined similarly when \(T\) is a collection of some ideals that
			are also dense left ideals. We denote by \(\mathbb{M}^r_R\) (resp. \(\mathbb{M}^l_R\)) the collection of all right (resp. left)
			Martindale sets of \(R\). 
		\end{definition}
		Clearly \(\mathbb{M}^r_R\) (or \(\mathbb{M}^l_R\)) is closed under  intersections.
		
		\section{Sub-Martindale rings of quotients}
		
		Suppose \(T\in \mathbb{M}^r_R\) is a right Martindale set of \(R\) and
		\[F=\{q\in Q^r_{\max}(R) | qA\subseteq R \text{ for some }A\in T\}.\]
		Let \(a,b\in F\). There are  \(A,B\in T\) such that \(aA\) and
		\(bB\subseteq R\). We note that \((a+b)(A\cap B)\subseteq aA + bB\subseteq R\).
		Furthermore \(ab(BA)\subseteq aRA\subseteq aA\subseteq R\). Obviously
		\(R\subseteq F\). Therefore we have \(F\) a right ring of quotients of \(R\).
		
		\begin{definition}
			Assume \(T\in \mathbb{M}^r_R\).
			We call the ring \[Q^r_T(R)=\left\{q\in Q^r_{\max}(R) \middle| qA\subseteq R, \text{  for some } A\in T \right\}\] a {\em \(T\)-Martindale right ring of quotients} or simply a {\em  sub-Martindale right ring of quotients}.  The {\em sub-Martindale left  ring of quotients}, \(Q^l_T(R)\), can be defined similarly when \(T\) is
			a left Martindale set.
		\end{definition}
		
		The notation \(Q^r_T(R)\) (resp. \(Q^l_T(R)\)) implicitly implies \(T\) is a right (resp. left) Martindale set of \(R\).
		Note if \(R\) is a semiprime ring and \(T\) is the set of all ideals that are also dense right ideals of
		\(R\), the ring \(Q^r_{T}(R)\) is precisely the well-known Martindale right ring of quotients, \(Q^r(R)\).

		\begin{lemma} We have	\(Q^r_T(R)=\left\{q\in Q^r_{\max}(R) \middle| qJ\subseteq R \text{  for some } J\in H \right\}\) whenever 
			 \(T\in \mathbb{M}^r_R\) and 	\(H=\left\{J  \leqslant _r R  \middle| I\subseteq J \text{  for some } I\in T \right\}\). 
		\end{lemma}
		
		\begin{proposition}\label{propMs}
			The ring \(Q^r_T(R)\) is a \(T\)-extension of \(R\).
		\end{proposition}
		\begin{proof}
			(MR.1) and (MR.2) are obvious, since \(Q^r_T(R)\subseteq Q^r_{\max}(R)\).  Now if
			\(\sigma \in Hom_R(A_R,R_R)\) where \(A\in T\), there is \(q\in Q^r_{\max}(R)\) such
			that  \(\sigma (a)=qa\) for all \(a\in A\). Note  \(qA=\sigma(A) \subseteq R\). This implies
			\(q\in Q^r_T(R)\). That is (MR.3).
		\end{proof}
		
		\begin{lemma}\label{subdenseidealhom}
			Let \(\sigma \in Hom_R(A_R,R_R)\) where \(A, B\unlhd_r^{den} R\) and \(B\subseteq A\).  Then \(\sigma(B)=0\)
			implies \(\sigma =0\).
		\end{lemma}
		\begin{proof}
			Assume \(\sigma(B)=0\) and  \(a\in A\). We have \(aB\subseteq B\). So \(\sigma (a)B=\sigma (aB)\subseteq
			\sigma(B)=0\). However \({\bf l}_R(B)=0\) since \(B \unlhd_r^{den} R\). Thus \(\sigma(a)=0\).
		\end{proof}
		
		It is well-known that elements of  \(Q_{\max}^r(R)\)  can be described as equivalence classes of the right \(R\)-linear functionals on dense right ideals of \(R\) (see \cite[p.~55]{Beidar96}, \cite[pp.~97--98]{Lambek66}, \cite[13.21]{Lam99}, \cite[p.~3]{Utumi56}). Martindale right rings of quotients can be characterized in a similar token \cite[14.9]{Lam99} for semiprime rings.  We show this result can be lift to sub-Martindale right rings of quotients. 
		\begin{definition}\label{idealfunctionpair}
			Suppose \(T\) is a right Martindale set of \(R\). Let \(\mathfrak{Q}\) be the collection of all ordered pairs \((A,\alpha)\) where \(A\in T\) and
			\(\alpha \in Hom_R(A_R,R_R)\). Define a relation \(\sim \) on \(\mathfrak{Q}\) such that
			\((A,\alpha)\sim (B,\beta)\)
			whenever \(\alpha(x)=\beta(x)\) if \(x\in A\cap B\).
		\end{definition}
		It is clear that \(\sim\) is  reflexive and symmetric.
		Suppose \((A,\alpha)\sim(B,\beta)\) and \((B,\beta)\sim (C,\gamma)\).
		We have \(\alpha(x)=\beta(x), \forall x \in A\cap B\) and \(\beta(x)=\gamma(x), \forall x\in B\cap
		C\). This implies \(\alpha(x)=\beta(x), \forall x\in A\cap B\cap C\). Clearly
		\(\alpha_{|A\cap C}-\gamma_{|A\cap C}\in Hom_R(A_R\cap C_R, R)\). Furthermore,
		\((\alpha_{|A\cap C} - \gamma_{|A\cap C})(A\cap B\cap C)=0\). By Lemma~\ref{subdenseidealhom} we have \((\alpha -
		\beta)(A\cap C)=0\). Thus \((A,\alpha)\sim (C,\gamma)\). Therefore \(\sim\) is an equivalence relation.
		For convenience we write \([A,\alpha]\) for \((A,\alpha)/\sim\).
		Let \([A,\alpha]\) and \([B,\beta] \in \mathfrak{Q}/\sim\). We define
		\begin{align*}
			[A,\alpha] + [B,\beta] &=[A\cap B, \alpha+\beta], \mbox{ and }\\
			[A,\alpha] \cdot [B,\beta] &=[BA, \alpha \circ \beta].
		\end{align*}
		A routine  computation shows that \((\mathfrak{Q}/\sim,+,\cdot)\) (or simply \(\mathfrak{Q}/\sim\))
		is indeed a ring.
		Note \([R,1]\) and \([R,0]\) are the unity and zero elements where \(1\) is
		the identity map and \(0\) is the zero map on \(R\).
		
		Define \(\theta :\mathfrak{Q}/\sim \ \rightarrow Q^r_T(R)\) such that \(\theta ([A,\alpha])
		=q\) where \(qx=\alpha(x)\) for \(x\in A\). This \(q\) is uniquely determined by (MR.3). Thus the map
		\(\theta\) is well-defined and can be shown to be a ring isomorphism.
		Furthermore the ring \(R\) can be identified as a subring of \(\mathfrak{Q}/\sim\) by the
		natural assignment \(r\mapsto [R,i_r]\) where \(i_r(x)=rx\) for all \(x, r\in R\). Therefore \(\theta([R,i_r])=r\).

		Immediately we have the following result.  
		\begin{theorem}\label{equiclass} The ring \(Q^r_T(R)\) can be viewed as equivalence classes of the right \(R\)-linear functionals on dense right ideals in \(T\), i.e. \(\mathfrak{Q}/\sim\).
		\end{theorem}
		To simplify the notation, whenever \(B\subseteq A\) we may write  \([A,\alpha]=[B,\alpha\)] instead of  \([A,\alpha]=[B,\alpha_{|B}]\)
		if the context is clear. Suppose \(A\) and \(B\) are extension rings of \(R\). We say \(A\) is  isomorphic  to \(B\) over \(R\) if there is a ring isomorphism \(f:A\rightarrow B\) such that \(f(r)=r\) for any \(r\in R\).

		\begin{theorem}\label{Textension}
			Let \(R\) be a subring of \(F\) and let \(T\) be a right Martindale set of \(R\).
			Then  \(F\) is isomorphic to \(Q^r_T(R)\)    over \(R\)  if and only if \(F\) is a \(T\)-extension of \(R\).
		\end{theorem}
		
		\begin{proof}
			(\(\Rightarrow\)) The fact that \(F\cong Q^r_T(R)\) over \(R\) implies \(F\) is a \(T\)-extension of \(R\) follows
			at once from Proposition~\ref{propMs}.
			
			(\(\Leftarrow\)) We prove the converse by showing \(F\) is isomorphic to  \(\mathfrak{Q}/\sim\) over \(R\).
			Let \(\alpha: F \rightarrow \mathfrak{Q}/\sim\) such that \(f \mapsto [D, \sigma_f]\) for some
			\(D\in T\) (by (MR.1)) and \(\sigma_f\in Hom_R(D_R,R_R)\) defined by \(\sigma_f(x)=fx\) for any \(x\in D\). If there is another \(D'\in T\) with \(\sigma'_f(x)=fx\) for any \(x\in D\), we have \(\sigma_f (d)=\sigma'_f(d)\) for any \(d\in D\cap D'\). Thus
			\([D,\sigma_f]=[D',\sigma'_f]\) and the mapping \(\alpha\) is well-defined. Suppose now \(\alpha(f_1)=[D_1,\sigma_{f_1}]\) and
			\(\alpha(f_2)=[D_2,\sigma_{f_2}]\). We have \([D_1,\sigma_{f_1}]+[D_2,\sigma_{f_2}]=[D_1\cap D_2,
			\sigma_{f_1}+\sigma_{f_2}]=[D_1\cap D_2,\sigma_{f_1+f_2}]\). Thus \(\alpha(f_1+f_2)=\alpha(f_1)+\alpha(f_2)\).
			Similarly we have \(\alpha(f_1\cdot f_2)=\alpha(f_1)\cdot \alpha(f_2)\) since
			\([D_1,\sigma_{f_1}]\cdot [D_2,\sigma_{f_2}]=[D_2 D_1,\sigma_{f_1}\circ \sigma_{f_2}]=[D_2 D_1,\sigma_{f_1f_2}]\). So \(\alpha\)
			is a ring homomorphism. Suppose \(\alpha(f_1)=\alpha(f_2)\). We have \([D_1,\sigma_{f_1}]=[D_2,\sigma_{f_2}]\). This implies
			\(f_1x=f_2x\) for all \(x\in D_1\cap D_2\), and thus \(f_1=f_2\) by (MR.2). Finally for \([D,\sigma]\in \mathfrak{Q}/\sim\) where
			\(\sigma \in Hom_R(D_R, R_R)\). From (MR.3), there is \(f\in F\) such that \(\sigma(x)=fx\) for all \(x\in D\). So \(\alpha(f)=[D, \sigma\)].
			By construction, \(\alpha\) fixes \(R\). We complete the proof.
		\end{proof}

		Note that \(I_i\) is a dense right ideal of \(R_i\) for each \(i\) if and only if \(\prod I_i\) is a dense
		right ideal of \(\prod R_i\). Furthermore  \(Q^r_{\max}(\prod R_i)=\prod Q^r_{\max}(R_i)\).
		\begin{theorem}\label{productquotient}Suppose \(\{R_k\}\) is a collection of rings. Let \(T_k\) be a right Martindale set of \(R_k\) for each \(k\). The collection 
			\(T=\{\prod I_k | I_k\in T_k\}\) is a right Martindale set of \(\prod R_k\). Then \(Q^r_T(\prod R_k)=\prod Q^r_{T_k}(R_k)\).
		\end{theorem}
		
		Denote by \(M_n(R)\) the \(n \times n\) matrix ring over \(R\). Obviously we have
		
		\begin{lemma}
			We have \(I\unlhd_r^{den} R\) if and only if \(M_n(I)\unlhd_r^{den} M_n(R)\).
		\end{lemma}
	
		Similar to \cite[2.3]{Utumi56}, we have
		\begin{theorem}\label{quotientmatrix} Let \(n\) be a positive integer. We have that 
			\(T\) is a right Martindale set of \(R\) if and only if \(T_M=
			\left\{M_n(I) \middle| I\in T\right\}\) is a right Martindale set of \(M_n(R)\). Furthermore, we have
			\[
			Q^r_{T_M}(M_n(R))=M_n(Q^r_T(R)).
			\]
		\end{theorem}
		\begin{proof} 
			The fact that \(T\) is a right Martindale set of \(R\) if and only if \(T_M=
			\left\{M_n(I) \middle| I\in T\right\}\) is a right Martindale set of \(M_n(R)\) follows easily from the previous lemma. We show the last equality.
			Recall \(Q^r_{\max}(M_n(R))=M_n(Q^r_{\max}(R))\) .

			Let \(f=\sum_{ij}f_{ij}E_{ij}\in Q^r_{T_M}(M_n(R))
			\subseteq M_n(Q^r_{\max}(R))\). Then
			we have \(fM_n(K)\subseteq M_n(R)\) for some \(K\in T\). That is
			\[(\sum_{ij}f_{ij}E_{ij})(\sum_{ij}k_{ij}E_{ij})=\sum_{ijl}f_{il}k_{lj}E_{ij}\in M_n(R)\]
			where \(k_{lj}\in K\). Equivalently \(f_{il}k_{lj}\in R\) for all \(i,j, l\). This implies
			\(f_{il}K\subseteq R\), and thus \(f_{il}\in Q^r_T(R)\) for all \(i\)  and \(l\). Therefore \(f\in {M_n}({Q^r_T}(R))\). This shows 
			\(Q^r_{T_M}(M_n(R))\subseteq M_n(Q^r_T(R))\).
			
		 Now suppose \(f=\sum_{ij}f_{ij}E_{ij}\in M_n(Q^r_T(R))\).
			For each \(f_{ij}\), there is \(I_{ij}\in T\) such that \(f_{ij}I_{ij}\subseteq R\).
			Let \(I=\cap_{ij} I_{ij}\in T\).  So \(f_{ij} I\subseteq R\) for all \(i,j\).
			In other words, we have \(f\in Q^r_{T_M}(M_n(R))\) and thus \(Q^r_{T_M}(M_n(R))\supseteq M_n(Q^r_T(R))\).
		\end{proof}

		\begin{theorem}\label{FDR_minimal}
			Let \(F=Q^r_T(R)\) and \(D\in T\) where \(T\in \mathbb{M}^r_R\) is a right Martindale set of \(R\).
			Suppose \(E=\left\{f \in F \middle| fD\subseteq D\right\}\).
			\begin{enumerate}
				\item[(1)]  \(End_R(D_R)\) is ring isomorphic to \(E\) over \(R\).
				\item[(2)] \(FD\subseteq D\) if and only if \(End_R(D_R)\) and \(F\) are ring isomorphic over \(R\). In particular if \(D\)
				is a minimal element of \(T\), then \(D\) is the unique minimal element of \(T\) and \(End_R(D_R)\cong F\) over \(R\).
				\item[(3)] If \({\bf l}_R(soc(R_R))=0\), then  \(End_R(soc(R_R))=Q_{\overline{T}_r}^r(R)\) where \(\overline T_{r}\) is the set of all ideals of \(R\)  that are also right dense.
			\end{enumerate}
		\end{theorem}
		\begin{proof}We show (1) and (2). The case (3) is an immediate consequence of (2).
			Recall \(F\) is a \(T\)-extension ring of \(R\). We note that \(E\) is a subring of \(F\).
			We are going to define a ring embedding \(\theta : End_R(D_R) \rightarrow F\). Let \(\alpha \in End_R(D_R)\). Then (MR.3) implies there is a \(f\in F\)
			such that \(\alpha(d)=fd\) for all \(d\in D\). Such an element \(f\) is uniquely determined by \(\alpha\) according to (MR.2). Define
			\(\theta (\alpha) = f\). Certainly \(\theta \) is well-defined.
			
			Let \(\alpha_1\) and \(\alpha_2\in End_R(D_R)\) such that \(f_1=\theta (\alpha_1)\) and \(f_2=\theta(\alpha_2)\). Note  \((f_1+f_2)d=f_1d+f_2d=\alpha_1(d)+\alpha_2(d)=(\alpha_1+\alpha_2)(d)\). That is \(\theta(\alpha_1+\alpha_2)= f_1+f_2=\theta(\alpha_1)+\theta(\alpha_2)\).
			
			If \(d\in D\), we have \(\alpha_1\alpha_2(d)=\alpha(f_2d)=f_1f_2d\) since \(f_2d\in D\). This implies \(\theta(\alpha_1\alpha_2)=
			f_1f_2=\theta(\alpha_1)\theta(\alpha_2)\) and thus \(\theta\) is a ring homomorphism. Suppose \(\theta (\alpha)=0\). That is \(\alpha (d)=0\) for all \(d\in D\). We have \(\alpha=0\) by (MR.2). Furthermore, the embedding \(R \hookrightarrow End_R(D_R)\) can be realized by  \(r\mapsto \alpha_r\) where
			\(\alpha_r(d)=rd \) for all \(d\in D\). By (MR.2) and (MR.3) \(\theta (\alpha_r)=r\) for any \(r\in R\).
			We can see that the image  \(\theta(End_R(D_R))\) is \(E\). Thus \(End_R(D)\) is isomorphic to \(E\) over \(R\). This shows (1). 
			
			We can view \(D\) as a subset of \(F\).  Note   \(FD\subseteq D\) is equivalently to \(E=F\). The first part of (2) follows immediately from (1). Finally
			we suppose \(D\) is a minimal element of \(T\). Let \(I\in T\). So \(I\cap D\in T\). The minimality of \(D\) implies \(I\cap D=D\), and
			thus \(D\) is the unique minimal element of \(T\). Let \(f\in F\). We have \(fD\subseteq R\) from (MR.1). Note that
			\(fDD\subseteq RD\subseteq D\). However \(DD=D\) since \(D\) the minimal element of \(T\) and \(DD\in T\). This forces \(fD\subseteq D\) and hence \({_F}D_R\) is a bimodule. This completes the proof of (2).
		\end{proof}

		\begin{lemma}\label{InclusionLemma}
			\begin{enumerate} 
				\item[(1)] \(Q^r_{T_1}(R)\subseteq Q^r_{T_2}(R)\) when \(T_1\subseteq T_2\) are
				right Martindale sets of \(R\).
				\item[(2)] Suppose \(\varnothing \neq {\mathcal A}\subseteq \mathbb{M}^r_R\). We have \(Q^r_{T_{\mathcal A}}(R)=\cap _{T\in {\mathcal A}}Q^r_T(R)\) where
				\(T_{\mathcal A}=\cap_{T\in {\mathcal A}} T\).
			\end{enumerate}
		\end{lemma}
		\begin{proof}
			Part (1) is Obvious.  Note \(\varnothing  \neq T_{\mathcal A}\in \mathbb{M}^r_R\). The inclusion \(Q^r_{T_A}(R)\subseteq \cap _{T\in \mathcal A}Q^r_T(R)\)
			follows from part (1). Now let \(f\in \cap _{T\in \mathcal A}Q^r_T(R)\). Thus there is
			\(I_j\in T_j\) for each \(T_j\in \mathcal{A}\) such that \(fI_j\subseteq R\). Let \(I=\sum_{j} I_j\). Clearly \(I\in T_j\) for each \(j\) and thus \(I\in T_{\mathcal A}\). Moreover \(fI\subseteq R\). Hence \(f\in Q^r_{T_{\mathcal A}}(R)\). This completes the proof of part (2).
		\end{proof}

		\begin{definition} Let \({\overline T}\) be the collection of all ideals of \(R\) that are right dense, i.e. the unique largest right Martindale set of \(R\). We call the ring \(Q_{\overline T}^r(R)\) the {\em maximal
			sub-Martindale right ring of quotients} of \(R\) and denote this special extension ring by  \(Q^r(R)\) if there is no confusion.  The maximal sub-Martindale left ring of quotients, \(Q^l(R)\), can be defined similarly.
		\end{definition}
Clearly that \(Q^r(R)\) is the unique largest sub-Martindale right ring of quotients of \(R\). Thus
	  \(Q^r_T(R)\subseteq Q^r(R)\) for any 
	 right Martindale set \(T\) of \(R\). Furthermore when \(R\) is semiprime, the ring  \(Q^r(R)\) is
		the well-known Martindale  right ring of quotients.

		\begin{lemma}\label{SinRinQmaxS}
			Assume \(S\) is a subring of \(R\) and \(I\leqslant_r  R\)  such that \(I\subseteq S\). Then \(I\leqslant_r ^{den} R\) if and only if  \(I\leqslant_r ^{den} S\) and \(S\subseteq R\subseteq Q^r_{\max}(S)\).
		\end{lemma}
		\begin{proof} (\(\Rightarrow\)) Assume \(I\leqslant_r ^{den} R\). Obviously we have \(I\leqslant_r ^{den} S\).		
			Let \(a,b\in R\) with \(a\neq 0\). There is a \(r_1\in R\) such that \(ar_1\neq 0\) and
			\(br_1\in I\). Furthermore, there is a \(r_2\in R\) such that \(a(r_1r_2)\neq 0\) and
			\(r_1r_2\in I\subseteq S\).  Thus \(b(r_1r_2)\in I\subseteq S\).			
			This shows \(S_S\leqslant_r ^{den} R_S\). We can view \(R\) as a subring of \(Q^r_{\max}(S)\).
	
			(\(\Leftarrow\)) Assume \(I\leqslant_r ^{den} S\) and \(S\subseteq R\subseteq Q^r_{\max}(S)\).  Let
			\(a,b\in R\) with \(a\neq 0\). There is \(s_1\in S\) such that \(0\neq as_1\in S\) since \(S_S\leqslant_r ^{den} R_S\). Similarly there is \(s_2\in S\) such that \(as_1s_2\neq 0\) and \(bs_1s_2\in S\). Now there is \(s_3\in S\) such that \(as_1s_2s_3\neq 0\) and \(bs_1s_2s_3\in I\) since
			\(I\leqslant_r ^{den} S\). As a consequence we have  \(I\leqslant_r ^{den} R\).
		\end{proof}
		
		\begin{definition}\label{MsetI}
			Let \(I\unlhd _r^{den}R\). Define \(\mathbb{M}^r_R(I)=\left\{T \in \mathbb{M}^r_R\middle| I \in T\right\}\) the collection of all right Martindale sets of \(R\) containing \(I\). The set \(\mathbb{M}^l_R(I)\) is defined similarly for \(I\unlhd _l^{den}R\). Suppose further that  \(S\) is a subring of \(R\) with \(I\subseteq S\). Define \(\Phi : \mathbb{M}^r_R(I) \rightarrow \mathbb{M}^r_S(I)\) such that
			\[\Phi(T)=\left\{K\unlhd S \middle | IJI\subseteq K \text{ for some }J\in T\right\} \text{ where } T\in \mathbb{M}^r_R(I).\]
		\end{definition}
		Note the previous lemma shows the mapping \(\Phi\) is well-defined. Furthermore if \(S=R\), we have \(\Phi (T)=T\).

		\begin{lemma}\label{IinSandR}
			Let \(S\) be a subring of \(R\) and \(I\unlhd_r^{den} R\) such that \(I\subseteq S\). Then the map \(\Phi : \mathbb{M}^r_R(I) \rightarrow \mathbb{M}^r_S(I)\) is a bijection. Suppose further that  \(\overline T_R\)  and   \(\overline T_S\) are the unique  maximum right  Martindale sets of \(R\) and  \(S\) respectively. Then
			we have \(\Phi(\overline T_R)=\overline T_S\).
		\end{lemma}
		\begin{proof}
			Suppose \(T_S\in \mathbb{M}^r_S(I)\) and let  \[T_R=\left\{L \unlhd R \middle | IJI \subseteq L \mbox{ for some }J \in T_S\right\}.\]  
			Note if \(J\in T_S\), we have \(IJI\unlhd_r^{den} S\) and \(IJI \unlhd R\). By Lemma~\ref{SinRinQmaxS}, we have \(IJI\unlhd_r^{den} R\). Therefore  \(T_R\in \mathbb{M}^r_R(I)\).

			Before we show \(\Phi\) is a bijection, we
			should note that \(IKI\in T_R\cap T_S\) if \(K\in T_R\cup T_S\). Now let \(K\in T_S\). Since \(I(IKI)I \subseteq K\),
			we have \(K\in \Phi(T_R)\). That is \(T_S\subseteq \Phi(T_R)\). On the other hand, if \(K\in \Phi(T_R)\), there is
			\(J\in T_R\) such that \(IJI\subseteq K\). There is also \(J'\in T_S\) such that \(IJ'I\subseteq J\). Thus
			\(I(IJ'I)I\subseteq K\). Since  \(IJ'I\in T_S\), we have \(K\in T_S\). As a consequence, we have \(\Phi(T_R)=T_S\)
			and thus \(\Phi\) is a surjection.
			
			Let \(T_R, T'_R\in \mathbb{M}^r_R(I)\) such that \(T_S=\Phi(T_R)=\Phi(T'_R)\in \mathbb{M}^r_S(I)\). Let \(J\in T_R\).
			So \(IJI\in T_S\). There is \(J'\in T'_R\) such that \(IJ'I\subseteq IJI\). Since \(IJ'I\in T'_R\), we have
			\(IJI\in T'_R\). So \(T_R\subseteq T'_R\). By symmetry we have \(T'_R\subseteq T_R\) and thus \(T_R= T'_R\). This shows
			\(\Phi\) is a bijection.
			
			We note that  \(\overline T_R\in \mathbb{M}^r_R(I)\) and \(\overline T_S\in \mathbb{M}^r_S(I)\). Since \(\Phi\) is onto, there is \(T\in \mathbb{M}^r_R(I)\) such that \(\Phi(T)=\overline T_S\). By construction, \(T\subseteq \overline T_R\).  Clearly we have  \(\overline T_S\subseteq \Phi(\overline T_R)\). The maximality of \(\overline T_S\) forces \(\overline T_S= \Phi(\overline T_R)\).
		\end{proof}

		\begin{theorem}\label{IsoMartQR}
			Let \(S\) be a subring of \(R\) and \(I\unlhd_r^{den} R\) such that \(I \subseteq S\). Let \(\Phi\) be the bijection introduced in Definition~\ref{MsetI}.
			\begin{enumerate}
				\item[(1)] If \(T_R\in \mathbb{M}^r_R(I)\), then \(Q^r_{T_R}(R)= Q^r_{\Phi(T_R)}(S)\), i.e. the following diagram commutes
				\[
				\begin{tikzcd}
					S\arrow[hook]{r}\arrow[hook]{d}&R \arrow[hook]{d}\arrow[hook]{ld}\\
Q^r_{\Phi(T_R)}(S) \arrow[Isom]{r} & Q^r_{T_R}(R).
				\end{tikzcd}
				\]
				\item[(2)] If \(T_S\in \mathbb{M}^r_S(I)\), then \(Q^r_{\Phi^{-1}(T_S)}(R)= Q^r_{T_S}(S)\).
			\end{enumerate}
			
		\end{theorem}
		\begin{proof}
			We first show (1).  Recall
			Lemma~\ref{SinRinQmaxS} shows  \(S\subseteq R\subseteq Q^r_{\max}(S)\), and  we have \(Q^r_{\max}(R)=Q^r_{\max}(S)\). Thus it suffices to show
			\(Q^r_{\Phi(T_R)}(S)\) is a \(T_R\)-extension of \(R\) by Theorem~\ref{Textension}. 	The conditions (MR.1) and (MR.2) are obvious since
			\(Q^r_{\Phi(T_R)}(S)\subseteq Q^r_{\max}(R)\). 
			
			Let \(K\in T_R\) and \(\sigma \in Hom_R(K_R,R_R)\). There is a unique \(f\in Q_{T_R}^r(R)\) such that \(\sigma (x)=fx\) for any \(x\in K\). Let \(J=IKI\subseteq K\). Note \(J\in T_R \cap \Phi(T_R)\). There is a unique \(g \in Q_{\Phi(T_R)}^r(S)\) such that \(\sigma_{|J}(x)=gx\)  for any \(x\in J\). 
			We note that both \(f\) and \(g\) are in \(Q^r_{\max}(R)=Q^r_{\max}(S)\). Then Lemma~\ref{subdenseidealhom} implies \(f=g\). This show (MR.3) is true and hence (1) is true.

Item (2) follows  from 	the fact that \(\Phi\) is a bijection.
		\end{proof}

		\begin{corollary}
			Let \(S\) be a subring of \(R\) and \(I\unlhd_r^{den} R\) such that \(I \subseteq S\).
			Then \(Q^r(S)=Q^r(R)\).
		\end{corollary}

		\section{Sub-Martindale symmteric rings of quotients}
		\begin{definition}\label{1compatible}
			Suppose \(A\in T_r\in \mathbb{M}_R^r\) and \(B\in T_l\in \mathbb{M}_R^l\). Let
			\(\sigma \in Hom_R(A_R,R_R)\) and \(\phi \in Hom_R({_R}B,{_R}R)\). The pair, \((A,\sigma)\)
			and \((B,\phi)\), is said to be {\em compatible} if there are 	\(\hat{A}\in T_r\) and \(\hat{B}\in T_l\) where 
\(\hat{A}\subseteq A\) and 		 \(\hat{B}\subseteq B\)   such
			that \(b\sigma (a)=(b)\phi a\) for any \(a\in \hat{A}\) and \(b\in \hat{B}\).
		\end{definition}
		
		The above concept of ``compatibility'' generalizes what is called {balanced} or {associativity condition} in Passman~\cite[p.~210]{Passman87} and \cite[p.~87]{Passman89}. See also Lanning~\cite[p.~49]{Lanning96}.
		
		For convenience, we stipulate that \(T_r\in \mathbb{M}_R^r\) and \(T_l\in \mathbb{M}_R^l\) hereafter.
		\begin{lemma}
			Suppose \((A,\sigma)\) and \((B,\phi)\) are compatible where \(A\in T_r\) and
			\(B\in T_l\).  Assume  \(A_0\in T_r\) and  \(B_0\in T_l\) such that   \(A_0\subseteq A\) and
			\(B_0\subseteq B\). Then \((A_0,\sigma_{|A_0})\) and
			\((B_0,\phi_{|B_0})\) are compatible.
		\end{lemma}
		\begin{proof} The compatibility of \((A, \sigma)\) and \((B,\phi)\) implies 
			there are \(\hat{A}\in T_r\) and \(\hat{B}\in T_l\) with \(\hat{A}\subseteq A\)
			and \(\hat{B}\subseteq B\) such that \(b\sigma a=b\phi a\) for any
			\(a\in \hat{A}\) and \(b\in \hat{B}\). In particular, this
			is true for \(a\in A_0\cap \hat{A}\in T_r\) and \(b\in B_0\cap \hat{B}\in T_l\).
		\end{proof}
		
		\begin{lemma}\label{compatibleLemma} Suppose
		 \(A_i\in T_r\), \(B_i\in T_l\), \(\sigma_i\in Hom_R({(A_i)}_R, R_R)\), and \(\phi_i\in Hom_R({_R}(B_i), {_R}R)\) where \(i=1, 2\). Assume \((A_1,\sigma_1)\)
			and \((B_1,\phi_1)\) are compatible.
			\begin{enumerate}
				\item[(1)] Suppose \((A_1,\sigma_1)\sim (A_2,\sigma_2)\).
				Then \((B_1,\phi_1)\sim (B_2,\phi_2)\) if and only if \((A_2,\sigma_2)\) and
				\((B_2,\phi_2)\) are compatible.
				\item[(2)] Suppose \((B_1,\phi_1)\sim (B_2,\phi_2)\). Then
				\((A_1,\sigma_1)\sim (A_2,\sigma_2)\) if and only if \((A_2,\sigma_2)\) and
				\((B_2,\phi_2)\) are compatible.
			\end{enumerate}
		\end{lemma}
		\begin{proof}
			We first prove  (1). Note that the discussion after Definition~\ref{idealfunctionpair} shows \(\sim\) is an equivalence relation. 
			From the previous lemma,
			we can assume, without loss of generality, that \(b\sigma_1 (a)=(b)\phi_1 a\) for any
			\(a\in A_1,b\in B_1\). Furthermore \((A_1,\sigma_1)\sim (A_2,\sigma_2)\) implies
			\(\sigma_1(a)=\sigma_2(a)\) and thus \(b\sigma_1 (a)=
			b\sigma_2 (a)\) for any \(a\in A_1\cap A_2\) and \(b\in B_1\cap B_2\).
			
			(\(\Rightarrow\)) Assume \((B_1,\phi_1)\sim (B_2,\phi_2)\). Thus \((b)\phi_1=(b)\phi_2\)
			for any \(b\in B_1\cap B_2\). Note \(b\sigma_1(a)=b\sigma_2(a)\) and
			\((b)\phi_1a=(b)\phi_2a\) for any \(a\in A_1\cap A_2\) and \(b\in B_1\cap B_2\).
			This implies \(b\sigma_2(a)=(b)\phi_2a\) for any \(a\in A_1\cap A_2\) and \(b\in B_1\cap B_2\).
			Since \(A_1\cap A_2\in T_r\) and \(B_1\cap B_2\in T_l\), we have \((A_2,\sigma_2)\) and
			\((B_2,\phi_2)\) are compatible.
			
			(\(\Leftarrow)\) Now suppose \((A_2,\sigma_2)\)  and \((B_2,\phi_2)\) are compatible.
			Without loss of generality we assume \(b\sigma_2(a)=(b)\phi_2a\) for any \(b\in B_2,
			a\in A_2\). This implies \((b)\phi_1 a=b\sigma_1(a)=b\sigma_2(a)=(b)\phi_2a\) for any
			\(a\in A_1\cap A_2\) and \(b\in B_1\cap B_2\). In other words,
			\(((b)\phi_1-(b)\phi_2)(A_1\cap A_2)=0\). Since \(A_1\cap A_2\in T_r\), we have
			\((b)\phi_1=(b)\phi_2\) for any \(b\in B_1\cap B_2\in T_l\) by Lemma~\ref{subdenseidealhom}. This implies 
			\((B_1,\phi_1)\sim (B_2,\phi_2)\) and thus (1) is true. 

 Item (2) can be obtained by symmetry.
		\end{proof}

		Recall an element of \(Q^r_{T_r}(R)\) can be identified as an equivalence class
		\([A,\sigma]\) for some \(A\in T_r\) and \(\sigma\in Hom_R(A_R,R_R)\). Similarly an
		element of \(Q^l_{T_l}(R)\) can be viewed as \([B,\phi]\) for some \(B\in T_l\) and
		\(\phi \in Hom_R({_R}B,{_R}R)\). Lemma~\ref{compatibleLemma} implies the notion of
		compatibility (see Definition~\ref{1compatible}) can be extended to the equivalence
		classes, i.e, elements of \(Q^r_{T_r}(R)\) or \(Q^l_{T_l}(R)\).
		
		\begin{definition}Suppose \(A\in T_r\), \(B\in T_l\) and \(\sigma\in Hom_R(A_R,R_R)\), \(\phi
			\in Hom_R({_R}B,{_R}R)\). The pair, \([A,\sigma]\) and \([B,\phi]\), is said to be
			{\em compatible} if there are \(A'\subseteq A, B'\subseteq B\) where \(A'\in T_r,B'\in T_l\) such that \(b\sigma (a)=(b)\phi a\) for \(a\in A',b\in B'\). Equivalently, the pair, \(f\in Q^r_{T_r}(R)\) and \(g\in Q^l_{T_l}(R)\), is {\em compatible} if \(bfa=bga\) for any \(a\in A, b\in B\) such that \(fA\subseteq R\) and \(Bg\subseteq R\) for some  \(A\in T_r,B\in T_l\). Furthermore, we say that \(g\) is a {\em left compatible} of \(f\) and \(f\) is a {\em right compatible} of \(g\).
		\end{definition}
		
		As a matter of fact,  we can assume without loss of generality that
		\(A'=A,B'=B\) in the definition above, since \([A',\sigma_{|A'}]=[A,\sigma]\) and
		\([B',\phi_{|B'}]=[B,\phi]\).
		
		\begin{lemma}\label{AddMul}
			Let \(A_i\in T_r\) and \(B_i\in T_l\) where \(i=1,2\).
			Suppose each pair, \([A_i,\sigma_i]\) and \([B_i,\phi_i]\), is compatible for \(i=1,2\). Then each of the following pairs is compatible:
			\begin{enumerate}
				\item[(1)] \([A_1,\sigma_1]+[A_2,\sigma_2]\) and \([B_1,\phi_1] + [B_2,\phi_2]\);
				\item[(2)] \([A_1,\sigma_1]\cdot[A_2,\sigma_2]\) and \([B_1,\phi_1] \cdot [B_2,\phi_2]\).
			\end{enumerate}
		\end{lemma}
		\begin{proof}
			Note \([A_1,\sigma_1]+[A_2,\sigma_2]=[A_1\cap A_2,\sigma_1+\sigma_2]\) and
			\([B_1,\phi_1]+[B_2,\phi_2]=[B_1\cap B_2,\phi_1+\phi_2]\). This implies
			\(b(\sigma_1+\sigma_2)(a)=(b)(\phi_1+\phi_2)a\) for any \(a\in A_1\cap A_2,
			b\in B_1\cap B_2\). This implies \([A_1,\sigma_1]+[A_2,\sigma_2]\) and \([B_1,\phi_1] + [B_2,\phi_2]\) are compatible. Thus (1) is true.
			
			Similarly \([A_1,\sigma_1]\cdot [A_2,\sigma_2]=[A_2A_1,\sigma_1\sigma_2]\) and
			\([B_1,\phi_1]\cdot [B_2,\phi_2]=[B_2B_1,\phi_1\phi_2]\). This implies
			\begin{align*}
			&b_2b_1(\sigma_1\sigma_2)(a_2a_1)=b_2b_1\sigma_1(\sigma_2(a_2a_1))=
			(b_2b_1)\phi_1\sigma_2(a_2a_1)\\
			&=(b_2(b_1)\phi_1)\sigma_2(a_2a_1)=
			(b_2(b_1)\phi_1)\phi_2)a_2a_1 =(b_2b_1)(\phi_1\phi_2)a_2a_1
			\end{align*}
			for any \(a_i\in A_i,
			b_i\in B_i\), \(i=1,2\). Therefore \(b(\sigma_1\sigma_2)(a)=(b)(\phi_1\phi_2)a\) for any
			\(a\in A_2A_1,b\in B_2B_1\). Case (2) follows at once.
		\end{proof}

		\begin{lemma}\label{compatibles}
			Suppose \(f\in Q^r_{T_r}(R)\) and  \(g\in Q^l_{T_l}(R)\). Then
			\begin{enumerate}
				\item[(1)] \(f\) and \(g\) are compatible if and only if there are \(A\in T_r, B\in T_l\) such that \(fA\cup Bg \subseteq R, bf=bg\) and \(fa=ga\) for any \(a\in A, b\in B\).
				\item[(2)] If \(g\) is left  compatible of \(f\), then \(g\) is uniquely determined. 
				\item[(3)] Suppose \(F\) is a subring of \(Q^r_{T_r}(R)\) and \(\chi : F \rightarrow Q^l_{T_l}(R)\) such that \(\chi(f) = e\) whenever \(f\) and \(e\)  are compatible. Then \(\chi\) is a ring embedding.
			\end{enumerate}
		\end{lemma}
		\begin{proof} Assume
			\(f\) and \(g\) are compatible. Then
			\((bf-bg)A=0\) for any \(b\in B\in T_l\) and \(A\in T_r\). Furthermore \(fA\cup Bg\subseteq R\). This implies \(bf=bg\) for any \(b\in B\) since \(A\) is right dense. Similarly \(fa=ga\) for any \(a\in A\). The rest of (1) follows easily. The second assertion (2) is an immediate consequence of Lemma~\ref{compatibleLemma}. Finally,  (3) is a consequence of Lemma~\ref{AddMul} and (2).
		\end{proof}

		\begin{theorem}\label{QSIso}
			The following rings are isomorphic over \(R\).
			\begin{enumerate}
				\item[(1)] \(F=\left\{f\in Q^r_{T_r}(R) \middle|  \exists B\in T_l \text{ such that } Bf\subseteq R\right\}\).
				\item[(2)] \(G=\left\{g\in Q^l_{T_l}(R) \middle| \exists A\in T_r  \text{ such that } gA\subseteq R\right\}\).
			\end{enumerate}
		\end{theorem}
		\begin{proof}It is clear that both \(F\) and \(G\) are  extension rings of \(R\).
			For each \(f\in F\), there are \(A\in T_r, B\in T_l\) such that \(fA\cup Bf \subseteq R\). Define \(\phi \in Hom_R({_R}B,{_R}R)\) by \((b)\phi= bf\). Recall \([B,\phi]\) can be viewed as an element of   \(Q^l_{T_l}(R)\).  Note that \((b)\phi=bg\) for any \(b\in B\) where \([B,\phi]\) is identified as \(g\in Q^l_{T_l}(R)\). Thus \(bga=bfa\) for any \(a\in A, b\in B\). We then have \(f\) and \(g\) are compatible. Furthermore, \(B(fa-ga)=0\) for any \(a\in A\). Thus \(fa=ga\) since \(B\) left dense. This forces \(g\in G\). Lemma~\ref{compatibles} shows that there is a ring isomorphism between \(F\) and \(G\) that maps an element in \(F\) to its compatible counterpart in \(G\), and the isomorphism fixes elements of \(R\). Hence \(F\) and \(G\) are isomorphic over \(R\).
		\end{proof}
		
		As noted in the theorem above there is a special subring of \(Q^r_{T_r}(R)\) which is isomorphic to a subring of \(Q^l_{T_l}(R)\) over \(R\). 
		
		\begin{definition}Given \(T_l\in \mathbb{M}_R^l, T_r\in \mathbb{M}_R^r\), we  define
			\[^{}_{T_l}Q^s_{T_r}(R) =\left\{f\in Q^r_{T_r}(R)\middle |  \exists B\in T_l \text{ such that } Bf\subseteq R\right\}.\] Or equivalently \[^{}_{T_l}Q^s_{T_r}(R) =
			\left\{g\in Q^l_{T_l}(R) \middle| \exists A\in T_r  \text{ such that } gA\subseteq R\right\}.\]
			We call \(^{}_{T_l}Q^s_{T_r}(R)\) a {\em \(\langle T_l,T_r\rangle\)-Martindale symmetric ring of quotients} or simply a {\em sub-Martindale symmetric ring of quotients}.
			
			If \(\overline T_l\) and \(\overline T_r\) are the unique maximal elements of \(\mathbb{M}_R^l\) and  \(\mathbb{M}_R^r\), respectively. We call {\em \(^{}_{\overline T_l}Q^s_{\overline T_r}(R)\) the maximal
			sub-Martindale symmetric ring of quotients}. In this case, we simply write \(Q^s(R)\) instead.
		\end{definition}

	  When \(R\) is semiprime, its
		maximal sub-Martindale symmetric ring of quotients is the well-known Martindale symmetric ring of quotients. Thus every ring has a unique largest sub-Martindale symmetric ring of quotients. We also note that 
\({}_{\{R\}}Q^s_{T_r}(R)={{}_{T_l}}Q^s_{\{R\}}(R)=R\).

		\begin{lemma}Suppose \(R\) is simple, (resp.\ prime, semiprime). Then so are \(Q^r_{T_r}(R)\), \(Q^l_{T_l}(R)\) and \(^{}_{T_l}Q^s_{T_r}(R)\).
		\end{lemma}
		\begin{proof}
			The lemma follows from the fact that if \(R\) is simple, (resp.\ prime, semiprime), so is every right ring of quotient of 	\(R\). (See \cite[p.380]{Lam99}.)
		\end{proof}
		
		\begin{lemma} Suppose \(R\) is simple. Then \(Q^r(R)=Q^l(R)=Q^s(R)=R\).
		\end{lemma}
		\begin{proof}The result follows easily from the fact that \(\{R\}\) is the only right (or left) Martindale set of \(R\).
		\end{proof}
		
		\begin{lemma}
			Suppose \(R\) is a domain. Then so are  \(^{}_{T_l}Q^s_{T_r}(R)\) and \(Q^s(R)\).
		\end{lemma}
		\begin{proof} Suppose \(R\) is a domain.
			Let \(f,g\in {^{}_{T_l}Q^s_{T_r}(R)}\) such that \(fg=0\). There are \(A\in T_r,B\in T_l\)
			with \(gA\cup Bf \subseteq R\). This implies \(bfga=0\) for any \(a\in A,b\in B\).
			Since \(bf,ga\in R\), we must have \(bf=0\) or \(ga=0\). If \(Bf=0\), we have \(f=0\) since \(B\) is left dense. Assume
			\(bf\neq 0\) for some \(b\in B\). We must have \(gA=0\) and thus \(g=0\). Hence \(^{}_{T_l}Q^s_{T_r}(R)\). In a similar token, \(Q^s(R)\) is a domain.
		\end{proof}

		\begin{definition} Suppose \(F\) is an extension ring of \(R\).
		We say \(F\) is a {\em \(\langle T_l,T_r\rangle\)-extension of \(R\)} where \(T_l\in \mathbb{M}_R^l\) and  \(T_r\in \mathbb{M}_R^r\) if the following
		conditions are true.
		\begin{description}
 \item[(MS.1)]  If \(f\in F\), then there are \(A\in T_r\) and \(B\in T_l\)
			such that \(fA \cup Bf\subseteq R\).
			\item[(MS.2)]  If \(f\in F\), \(A\in T_r\), and \(B\in T_l\), then \(f=0\) whenever
			\(fA=0\) or \(Bf=0\).
			\item[(MS.3)]  If \(\sigma\in Hom_R(A_R,R_R)\) and \(\phi\in Hom_R({_R}B,{_R}R)\) where 
	\(A\in T_r\) and \(B\in T_l\)		
			such that \(b\sigma (a)=(b)\phi a\) for any
			\(a\in A, b\in B\), then there is an \(f\in F\) such that \(\sigma (a)=fa\)
			and \((b)\phi =bf\) for any \(a\in A\) and \(b\in B\).
		\end{description}
		\end{definition}
		Note that condition (MS.3) implies \([A,\sigma]\) and \([B,\phi]\) are
		compatible.
		
		\begin{theorem}
			Let \(R\) be a subring of \(F\). Then \(F\) is isomorphic to \({^{}_{T_l}Q^s_{T_r}}(R)\) over \(R\)
			if and only if \(F\) is a \({\langle} T_l,T_r\rangle\)-extension of \(R\).
		\end{theorem}
		\begin{proof} We assume \(\sigma \in Hom_R(A_R,R_R)\), \(\phi\in Hom_R({_R}B,{_R}R)\) where \(A\in T_r, B\in T_l\).
			
			\((\Rightarrow)\) Without loss of generality, we can assume \(F= {^{}_{T_l}Q^s_{T_r}}(R)\). Thus
			(MS.1) is obvious.   For (MS.2), we assume \(f\in F\). If \(fA=0\), we have \(f=0\) since   \({^{}_{T_l}Q^s_{T_r}}(R)\subseteq Q^r_{\max}(R)\). Now suppose
			\(Bf=0\). So we have \(BfA=0\). Since \(B\in T_l\), we have \(fA\subseteq {\bf r}_R(B)=0\). 
			This implies \(fA=0\), and thus \(f=0\). So (MS.2) is true. Now assume \(b\sigma (a)=(b)\phi a\) for any \(a\in A,b\in B\). This
			implies that there is an \(f\in Q^r_{T_r}(R)\) such that \(\sigma (a)= fa\) for any
			\(a\in A\). Therefore \((b)\phi a=b\sigma (a)= bf a\) for any \(a\in A,b\in B\).
			This forces \(((b)\phi - bf)A=0\). Again \(A\in T_r\) implies \((b)\phi = bf\) for
			any \(b\in B\). Since \((B)\phi \subseteq R\), we have \(Bf\subseteq R\), i.e.\
			\(f\in {^{}_{T_l}}Q^s_{T_r}(R)\). We have (MS.3).
			
			\((\Leftarrow)\) Assume \(F\) is a \({\langle} T_l,T_r\rangle\)-extension of \(R\).
			For any \(f\in F\), there are \(A\in T_r\) and \(B\in T_l\) such that \(fA\cup Bf\subseteq R\) by (MS.1). Suppose \(\sigma_f \in Hom_R(A_R,R_R),
			\phi_f \in Hom_R({_R}B,{_R}R)\) where \(\sigma_f(a)=fa, (b)\phi_f=bf\) for any \(
			a\in A, b\in B\). 
			Let \(\alpha : F\rightarrow {^{}_{T_l}Q^s_{T_r}}(R)\) such that \(\alpha (f)=\left[A,\sigma_f \right]\).
			Similar to the proof of Theorem~\ref{Textension}, we can show the map \(\alpha\) is a ring isomorphism that fixes \(R\).
		\end{proof}
		
A bijection \(\tau :R\rightarrow S\) is a ring  anti-isomorphism if \(\tau(a+b)=\tau(a)+\tau(b)\), \(\tau(ab)=\tau(b)\tau(a)\), and it preserves the unity. Thus $S$ is isomorphic to $R^{op}$ the opposite ring of $R$. It is well-known that for any ring $R$, there is a cannonical anti-isomorphism between \(Q^r_{\max}(R)\) and \(Q^l_{\max}(R^{op})\). 
		\begin{lemma}
			Suppose \(\tau : R\rightarrow S\) is a ring anti-isomorphism. Define
			\(T^{(\tau)}=\left\{\tau(I)\middle|I \in T\right\}\) whenever
			\(T\in \mathbb{M}^r_R \cup \mathbb{M}^l_R\).
			Then \(\tau\) extends uniquely to a ring anti-isomorphism for each of the following.
			\begin{enumerate}
				\item[(1)] \(Q^r_{T_r}(R)\rightarrow Q^l_{T_r^{(\tau)}}(S)\).
				\item[(2)] \({^{}{_{T_l}}}Q^s{_{T_r}}(R)\rightarrow {^{}_{T_r^{(\tau)}}}Q^s_{T_l^{(\tau)}}(S)\).
			\end{enumerate}
		\end{lemma}
		\begin{proof}The proof is similar to the case \(Q^r_{\max}(R)\) and \(Q^l_{\max}(R^{op})\). 
		\end{proof}

		\begin{corollary}
			Let \(E\) be a free right \(R\)-module with a finite basis of \(n\) elements. There
			is an anti-isomorphism \(\tau : End_R(E_R) \rightarrow M_n(R)\). Furthermore the following pairs are anti-isomorphic.
			\begin{enumerate}
				\item[(1)] \(Q^r_{\max}(End_R(E_R))\) and \(Q^l_{\max}(M_n(R))\).
				\item[(2)] \(Q^r(End_R(E_R))\) and \(Q^l(M_n(R))\).
				\item[(3)] \(Q^s(End_R(E_R))\) and \(Q^s(M_n(R))\).
			\end{enumerate}
		
		\end{corollary}

		Recall a map \(x\mapsto x^*\) of a ring \(R\) into itself is an involution if it is
		a ring anti-automorphism such that \(x^{**}=x\) for any \(x\in R\). A map $D:R\rightarrow S$ is a derivation if $D(xy)=D(x)y+xD(y)$.
		\begin{corollary}
			Suppose \(R\) is a ring with involution \(*\). Then \(*\) extends to an involution
			on \(Q^s(R)\).
		\end{corollary}
		
		\begin{lemma}\label{UniqueExtensionsHomSym}
			Every automorphism (resp. derivation) of \(R\) extends uniquely to an automorphism (resp. derivation) of \({^{}_{T_l}}Q^s_{T_r}(R)\). The mappings also extend uniquely to \(Q^r_{T_r}(R)\) (and  to  \(Q^l_{T_l}(R)\)).
		\end{lemma}
		\begin{proof}
			The proof is similar to that of~\cite[Lemma 10.9]{Passman89}. We show the case for \({^{}_{T_l}}Q^s_{T_r}(R)\). The existence of an extended automorphism is obvious. We show
			the uniqueness.
			Let \(\alpha, \alpha'\) be two automorphisms of \({^{}_{T_l}}Q^s_{T_r}(R)\) such that
			\(\alpha_{|R}=\alpha'_{|R}\). Let \(f\in {^{}_{T_l}}Q^s_{T_r}(R)\).
			There is \(A\in T_r\) such that
			\(fA\subseteq R\). For any \(a\in A\), we have
			\(\alpha(f)\alpha(a)=\alpha(fa)=\alpha'(fa)=\alpha'(f)\alpha'(a)=
			\alpha'(f)\alpha(a)\). Therefore \((\alpha(f)-\alpha'(f))A=0\). (MS.2) implies \(\alpha(f)=\alpha'(f)\).
			
			Now we show the existence of an extended derivation. Suppose \(\delta :R\rightarrow R\) is a derivation. Let \(f\in {^{}_{T_l}}Q^s_{T_r}(R)\).
			There are \(A\in T_r, B\in T_l\) and  \(\sigma\in Hom_R(A_R,R_R), \phi\in
			Hom_R({_R}B,{_R}R)\) such that \(\sigma (a)=fa, (b)\phi=bf\) for any \(a\in A,b\in B\).

			Note that \(\delta(A^2)\subseteq A\) and thus \(f\delta(\bar a)\in fA\subseteq R\). This implies \(\delta(f\bar a) - f\delta(\bar a)\in R\). The mapping  \(\alpha_f :A^2
			\rightarrow R\) such that \(\alpha_f(\bar a)=\delta(f\bar a) - f\delta(\bar a)\)
			for \(\bar a\in A^2\) is well-defined.
			
			Note \(\alpha_f(\bar a_1+\bar a_2)=
			\delta(f(\bar a_1+\bar a_2))-f(\delta(\bar a_1+\bar a_2))=\delta(f\bar a_1)-
			f\delta (\bar a_1) + \delta (f\bar a_2)-f\delta(\bar a_2)=\alpha_f(\bar a_1)+
			\alpha_f(\bar a_2)\). Furthermore, \(\alpha_f(\bar a_1 \bar a_2)=\delta (f\bar a_1\bar a_2)- f\delta(\bar a_1\bar a_2)=\delta((f\bar a_1)\bar a_2)-f[\bar a_1\delta(\bar a_2)+\delta(\bar a_1)\bar a_2]=f\bar a_1\delta (\bar a_2)+\delta(f\bar a_1)\bar a_2-f\bar a_1\delta(\bar a_2)-f\delta(\bar a_1)\bar a_2=\alpha_f(\bar a_1)\bar a_2\). It is clear that \(\alpha_f(0)=0\), and thus
			\(\alpha_f\in Hom_R({(A^2)}_R,R_R)\). Similarly we can define \(\beta_f \in Hom_R({_R}(B^2),{_R}R)\) such that  \(\beta_f(\bar b)=\delta(\bar b f) - \delta(\bar b)f\) for \(\bar b\in B^2\). Furthermore, \((\bar b) \beta_f\bar a-\bar b \alpha_f(\bar a)=\delta(\bar b f)\bar a - \delta(\bar b)f\bar a-\bar b\delta(f\bar a)+\bar bf\delta(\bar a)= 0\). By (MS.3), there exists a unique \(q_f\in {^{}_{T_l}}Q^s_{T_r}(R)\) such that \(\alpha_f(\bar a)=q_f \bar a, (\bar b)\beta_f=\bar b q_f\) for any \(\bar a\in A^2, \bar b\in B^2\).

			Define \(\bar \delta: {^{}_{T_l}}Q^s_{T_r}(R) \rightarrow {^{}_{T_l}}Q^s_{T_r}(R)\) such that 
			\(\bar\delta(f) = q_f\) when \(f\in {^{}_{T_l}}Q^s_{T_r}(R)\). We are to show
			\(\bar \delta \) is the desired extension of \(\delta\). We note that
			\(\delta (f\bar a)= \alpha_f(\bar a)+f\delta(\bar a)=q_f(\bar a)+f\delta(\bar a)=
			\bar \delta (f)\bar a + f \delta(\bar a)\) for any \(\bar a\in A^2\).
			Assume further \(g\in {^{}_{T_l}}Q^s_{T_r}(R)\) and,
			without loss of generality, \(fA \cup gA\subseteq R\). Since \(fA^3\subseteq A^2\) and
			\(fgA^2\subseteq R\), it follows that for any \(a\in A^4, \bar \delta(fg)a=\delta(fga)-
			fg\delta(a)= \bar \delta(f)ga+f\delta(ga)-fg\delta(a)=f\delta(ga)+\bar \delta(f)ga-
			fg\delta(a)=f\delta(ga)+\bar \delta(f)ga-fg\delta(a)=f\bar \delta(g)a+\bar \delta(f)ga\). Thus \(\bar \delta (fg)=f\bar \delta(g)+ \bar \delta(f)g\). When
			\(r\in R\), \(\delta(r)a+r\delta(a)=\delta (ra)=\bar \delta(r)a+r\delta(a) \). This means \((\delta (r)-\bar \delta (r))A^4=0\). We have \(\bar \delta = \delta \) on \(R\).
			Now we show the extended derivation is unique.
			Suppose \(\bar \delta'\) is another extension of \(\delta\). Since \(fA\subseteq R\), we have \(\bar \delta(f)a+f\delta(a)=\bar \delta(fa)=\delta(fa)=\bar \delta'(fa)=\bar \delta'(f)a+f\delta(a)\), for \(a\in A\). This implies \((\bar \delta(f)-\bar \delta'(f))A=0\), and hence \(\bar\delta=\bar \delta'\).
		\end{proof}
		
		The following results are similar to Theorems~\ref{productquotient} and~\ref{quotientmatrix}.
		\begin{theorem} Given that \(\{R_k\}\) is a collection of rings, we
			let \(T^k_l\in \mathbb{M}_{R_k}^l\) and \(T^k_r \in \mathbb{M}_{R_k}^r\)  for each \(k\). Suppose further that
			\(T_l= \{\prod I_l^k|I_l^k \in T^k_l\}\) and \(T_r= \{\prod I_r^k|I_r^k \in T^k_r\}\). Then \(^{}_{T_l}Q^s_{T_r}(\prod R_k)=\prod {^{}_{T^k_l}Q^s_{T^k_r}(R_k)}\).
		\end{theorem}
		\begin{theorem} Let \(n\) be a positive integer.
			Assume  \(T_l\in \mathbb{M}_R^l\) and  \(T_r\in \mathbb{M}_R^r\). We define \(T'_x=
			\left\{M_n(I) \middle| I\in T_x\right\}\) where  \(x=l \text{ or } r\). Then
			\[
			^{}_{T'_l}Q^s_{T'_r}(M_n(R))=M_n(^{}_{T_l}Q^s_{T_r}(R)).
			\]
		\end{theorem}

		\section{Center}
		\begin{definition}
			Given a ring \(R\), the {\em center} of \(R\) is \[Z(R)=\left\{r\in R\middle | rx=xr \text{ for any } x\in R\right\}.\]
		\end{definition}

		\begin{lemma}\label{centerTran}
			Suppose \(R\), \(E\), and \(F\) are subrings of  \(Q^r_{\max}(R)\) such that \(R\subseteq E\subseteq F\subseteq Q^r_{\max}(R)\). Then
			\(Z(R)\subseteq Z(E) \subseteq Z(F) \subseteq Z(Q^r_{\max}(R))\).
		\end{lemma}
		\begin{proof}
			The result follows from the fact that \[Z(Q^r_{\max}(R))=\{q\in Q^r_{\max}(R)| qx=xq  \text{ for any } x\in R\}.\] 
		\end{proof}
		
		\begin{corollary}Suppose \(T_r,T^{'}_r\in \mathbb{M}^r_R\) and \(T_l,T^{'}_l\in \mathbb{M}^l_R\).
			\begin{enumerate}
				\item[(1)] If \(T^{'}_r\subseteq T_r\), then  \(Z(Q^r_{T^{'}_r}(R))\subseteq Z(Q^r_{T_r}(R))\).
				\item[(2)] If \(T^{'}_r\subseteq T_r\) and \(T^{'}_l\subseteq T_l\), then \(Z(^{}_{T^{'}_l}Q^s_{T^{'}_r}(R))\subseteq Z(^{}_{T^{}_l}Q^s_{T^{}_r}(R))\).
			\end{enumerate}
		\end{corollary}
		\begin{lemma}\label{CenterSym}
			Let \(f\in Z(^{}_{T_l}Q^s_{T_r}(R))\). There exists \(X\in T_l\cap T_r\) such that
			\(fX\cup Xf\subseteq R\).
		\end{lemma}
		\begin{proof}It suffices to show there is \(X\in T_l\cap T_r\) such that \(fX\subseteq R\) since \(f\in Z(^{}_{T_l}Q^s_{T_r}(R))\). Note that there are \(A\in T_r\) and \(B\in T_l\) such that \(fA \subseteq R\) and \(Bf\subseteq R\). Let \(X=A+B\). Obviously \(X\in T_l\cap T_r\). Furthermore, \(fX\subseteq fA+fB=fA+Bf\subseteq R\).
		\end{proof}
		
		\begin{theorem}\label{centerLBi}
			Suppose \([A,\sigma]\in Q^r_T(R)\). Then
			\([A,\sigma]\in Z(Q^r_T(R))\) if and only if  \(\sigma\) is an \((R,R)\)-bimodule homomorphism. 
		\end{theorem}
		\begin{proof} For each \([A,\sigma]\in Q^r_T(R)\), there is an \(f\in Q^r_T(R)\) such that \(\sigma (a)=fa\) for all \(a\in A\).

(\(\Rightarrow\))  If \([A,\sigma]\in Z(Q^r_T(R))\), we have 	\(fr=rf\) for any \(r\in R\). Thus \(\sigma (ra)=fra=rfa=r\sigma (a)\) when \(r\in R, a\in A\). This shows \(\sigma\) is an \((R,R)\)-bimodule homomorphism.
			
			(\(\Leftarrow\)) Now suppose the contrary that \(\sigma:A\rightarrow R\) is an \((R,R)\)-bimodule homomorphism. We are to show the associated \(f\) is in the center of \(Q^r_T(R)\). Note \(rfa=r\sigma (a) = \sigma (ra) =fra\) where \(r\in R, a\in A\). This implies \((rf-fr)A=0\) and thus \(rf=fr\) for any \(r\in R\). That is \(f\in Z(Q^r_T(R))\).
		\end{proof}

		\begin{theorem}\label{CenterQs}
			Suppose \([A,\sigma]\in {^{}_{T_l}Q^s_{T_r}(R)}\) where \(A\in T_r\). Then
			\([A,\sigma]\in Z(^{}_{T_l}Q^s_{T_r}(R))\) if and only if \(\sigma\) is an \((R,R)\)-bimodule homomorphism, i.e. \(\sigma\in Hom_R({_R}A_R,{_R}R_R)\).
		\end{theorem}
		\begin{proof}It is a consequence of Lemma~\ref{centerTran} and Theorem~\ref{centerLBi}. 
		\end{proof}

		\begin{theorem}
			Let \(T_s=\left\{I\unlhd R \middle|I\unlhd_r^{den} R, I\unlhd_l^{den} R\right\}\). We have
			\[Z(Q^r_{T_s}(R))=Z(^{}_{T_s}Q^s_{T_s}(R))=Z(Q^l_{T_s}(R))=Z(Q^s(R)).\]
		\end{theorem}
		\begin{proof}Clearly \(T_s\) is both a left and a right  Martindale  set of \(R\). Since \(R\in T_s\),  we have \(T_s\neq \varnothing\).
			Theorem~\ref{QSIso} shows that if  \([A,\sigma]\in Q^r_{T_s}(R)\) and \([B,\phi]\in Q^l_{T_s}(R)\) are compatible, then both \([A,\sigma]\) and \([B,\phi]\) can be viewed as the same element of \(^{}_{T_s}Q^s_{T_s}(R)\). We also note that \(T_s\) is the unique maximal element
			of \(\mathbb{M}^r_R\cap \mathbb{M}^l_R\). 

We first show \(Z(Q^r_{T_s}(R))=Z(^{}_{T_s}Q^s_{T_s}(R))\). 			
			Lemma~\ref{centerTran} shows that \(Z(^{}_{T_s}Q^s_{T_s}(R))\subseteq Z(Q^r_{T_s}(R))\). Now
			suppose \([A,\sigma]\in Z(Q^r_{T_s}(R))\). Theorem~\ref{centerLBi} implies that
			\(\sigma\) is an \((R,R)\)-bimodule homomorphism. We have that  \([A,\sigma]\in
			Z({^{}_{T_s}Q^s_{T_s}}(R))\) since 
			\([A,\sigma]\in {^{}_{T_s}Q^s_{T_s}}(R)\). 
			Therefore 	\(Z(Q^r_{T_s}(R))=Z(^{}_{T_s}Q^s_{T_s}(R))\).
			 Similarly we have \(Z(Q^l_{T_s}(R))= Z(^{}_{T_s}Q^s_{T_s}(R))\).  
			
			Next, we have \(Z(^{}_{T_s}Q^s_{T_s}(R))\subseteq Z(Q^s(R))\) by Lemma~\ref{centerTran}. Furthermore,  we  have \(Z(Q^s(R)) \subseteq  {}^{}_{T_s}Q^s_{T_s}(R)\) by Lemma~\ref{CenterSym}. This shows the third equality.
		\end{proof}
		
		\begin{theorem}\label{centers}
			\(Z(Q^r(R))\cap Z(Q^l(R))=Z(Q^s(R))\subseteq Z(Q^r_{\max}(R))=Z(Q^r(R))\).
		\end{theorem}
		\begin{proof}Lemma~\ref{centerTran} implies \(Z(Q^s(R))\subseteq  Z(Q^r(R))\subseteq Z(Q^r_{\max}(R))\).
			Suppose \(q\in Z(Q^r_{\max}(R))\). It is easy to see \(q^{-1}R=\{r\in R \mid qr\in R\}\unlhd_r^{den} R\).
			Moreover \(q(q^{-1}R)\subseteq R\). Thus \(q\in Q^r(R)\). And hence \(q\in Z(Q^r(R))\). This means \(Z(Q^s(R))\subseteq Z(Q^r_{\max}(R))=Z(Q^r(R))\). By symmetry, we have \(Z(Q^s(R))\subseteq Z(Q^r(R))\cap Z(Q^l(R))\). 
			
			Suppose \(f\in Z(Q^r(R))\cap Z(Q^l(R))\). There are \(A\unlhd_r^{den} R\) and \(B\unlhd_l^{den} R\) such that \(fA\cup Bf\subseteq R\), and thus \(f\in Q^s(R)\). This implies \(f\in Z(Q^s(R))\).
		\end{proof}
\begin{corollary}
			Suppose \(Q^l_{\max}(R)=Q^r_{\max}(R)\). Then \[Z(Q^s(R))=Z(Q^r(R))=Z(Q^r_{\max}(R)).\]
\end{corollary}

		\begin{lemma}
			Let \(T_s=\{I\unlhd R |I\unlhd_r^{den} R, I\unlhd_l^{den} R\}\).
			When \(R\) is semiprime, We have
			\(Q^r(R)=Q^r_{T_s}(R)\), \(Q^s(R)={^{}_{T_s}Q^s_{T_s}(R)}\), and \(Q^l(R)=Q^l_{T_s}(R)\). 
		\end{lemma}

		\section{Examples}
		In this section  \(k\) is a field and \(p\) is a prime number.
		
		\begin{example}
			Suppose \(R=M_2(\mathbb{Z})\), the full \(2\times 2\) matrix ring over \(\mathbb{Z}\).
			Thus \(R\) is a prime Goldie ring, and \(Q^r_{\max}(R)=Q^r_{cl}(R)=M_2(\mathbb{Q})\). The unique maximal right Martindale set of \(\mathbb{Z}\) is
			\[\overline T =\left\{n\mathbb{Z} \middle| n=1,2,\ldots\right\}.\]
			Therefore the unique maximal right Martindale set of \(R\) is \[\left\{M_2(n\mathbb{Z})\middle | n=1,2,\ldots\right\}.\]

			Routine  computations show that  the maximal sub-Martindale right ring of quotients of \(R\) is \(Q^r(R)=M_2(\mathbb{Q})\). Similarly  \(Q^l(R)=M_2(\mathbb{Q})\) and \(Q^s(R)=M_2(\mathbb{Q})\). 
The set
			\[T_{p}=\left \{M_2(p^n\mathbb{Z}) \middle| n \text{ is a non-negative integer}\right \}\] is both a right and a left Martindale set of \(R\), and thus
			\[Q^r_{T_{p}}(R)=\left \{\frac{1}{p^n}q \middle| q\in R \text{ and } n \text{ a non-negative integer}\right \}.\]
			Therefore we have \(R\subsetneq Q_{T_{p}}^r (R)= {{^{}_{T_{p}}}Q^s_{T_{p}}(R)}
			\subsetneq Q^r(R)\).
		\end{example}

		\begin{example}\label{rightMartindaleEgs}
			
			Let \(R=\begin{pmatrix} k & k \\ 0 & k \end{pmatrix}\). We have \(Q^r_{cl}(R)=R\) and
			\(Q^r_{\max}(R)=M_2(k)\) (see~\cite[13.13]{Lam99}). Note  that \(R\) is not semiprime. Using \cite[1.17]{Lam91} we have \[\left\{\begin{pmatrix} 0 & k \\ 0 & 0\end{pmatrix} ,
			\begin{pmatrix} k & k \\ 0 & 0\end{pmatrix} , \begin{pmatrix} 0 & k \\ 0 & k\end{pmatrix} ,\begin{pmatrix} k & k \\ 0 & k\end{pmatrix} 
			\right\}\] the collection of all non-zero ideals of \(R\). 
			The maximal right and the maximal left Martindale sets of \(R\) are respectively \[\overline T_r=\left \{
			\begin{pmatrix} 0 & k \\ 0 & k\end{pmatrix}  , \begin{pmatrix} k & k \\ 0 & k\end{pmatrix} \right\} \text{ and }  
			\overline T_l=\left \{\begin{pmatrix} k & k \\ 0 & 0\end{pmatrix} ,\begin{pmatrix} k & k \\ 0 & k\end{pmatrix} 
			\right\}.   
			\] 
		
			A routine computation shows \(Q^r(R)= Q^s(R)=Q^r_{\max}(R)=M_2(k)\).  Furthermore \(D=\begin{pmatrix} 0 & k \\ 0 & k\end{pmatrix} \) is the minimal element of \(\overline T_r\). Therefore we have \(End_R(D_R)=Q^r(R)\) by Theorem~\ref{FDR_minimal}.
		\end{example}
		\begin{example}
			Let \(R=\begin{pmatrix}  \mathbb{Z} & \mathbb{Q} \\ 0 & \mathbb{Q} \end{pmatrix} \).
			We have \(Q^r_{\max}(R)=M_2(\mathbb{Q})\). The collection of all non-zero ideals of \(R\) is \[\left \{ \begin{pmatrix} n\mathbb{Z} & \mathbb{Q} \\ 0 &
				m\mathbb{Q}\end{pmatrix} \middle | n=0,1,2,\ldots; m=0,1\right \}.\]
			We have the maximal right and the maximal left Martindale sets of \(R\) are respectively
			\begin{align*}
			\overline T_r=&\left \{ \begin{pmatrix} n\mathbb{Z} & \mathbb{Q} \\ 0 &
				\mathbb{Q}\end{pmatrix} \middle | n=0,1,2,\ldots\right \} \text{ and } \\ \overline T_l=&\left \{ \begin{pmatrix} n\mathbb{Z} & \mathbb{Q} \\ 0 &
				m\mathbb{Q}\end{pmatrix} \middle | n=1,2,3,\ldots ; m=0, 1\right \}.\end{align*}
				
			We also note that \(T_p =\left\{ \begin{pmatrix} p^n\mathbb{Z} & \mathbb{Q} \\ 0 & 		\mathbb{Q}\end{pmatrix}  \middle| n=0,1,2,\ldots \right\}\) is both a right and a left Martindale set of \(R\).	
			Therefore, \begin{multline*}{^{}_{T_{p}}Q^s_{T_{p}}(R)}=Q_{T_p}^r(R)=\left\{ \begin{pmatrix} \dfrac{a}{p^n} & q_1 \\ 0 & 	q_2\end{pmatrix}  \middle| a\in \mathbb{Z}, q_i\in \mathbb{Q}, n=0,1,2,\ldots \right\} \\
			\subsetneq Q^s(R)=\begin{pmatrix} \mathbb{Q} & \mathbb{Q} \\ 0 &
				\mathbb{Q}\end{pmatrix} \subsetneq Q^r(R)=M_2(\mathbb{Q}).
\end{multline*}
		\end{example}
		\begin{example}
			Let \(R=\begin{pmatrix} k& k \bigoplus k\\ 0 &
				k\end{pmatrix} \cong \left\{ \begin{pmatrix} a& 0 & b\\ 0 &
				a & c\\ 0&0&d\end{pmatrix}    \middle | a,b,c,d\in k \right\}\). Note that \(Q_{max}^r(R)=M_3(k)\). The collection of all non-zero ideals of \(R\) is
		\[
		\left\{
			I_1=\begin{pmatrix} 0& 0 & k\\ 0 &
				0& k\\ 0&0&0
				\end{pmatrix},
			I_2=\begin{pmatrix} 0& 0 & k\\ 0 &
				0& k\\ 0&0&k
				\end{pmatrix},  
			I_3=\left\{\begin{pmatrix} a& 0 & k\\ 0 &
				a& k\\ 0&0&0\end{pmatrix}  \middle | a\in k\right\},
			I_4=R
			\right\}.\]
			The maximal right and the maximal left Martindale sets of \(R\) are respectively
			\[\overline T_r=\left\{I_2, I_4 \right\} \text{ and } \overline T_l=\left\{I_3, I_4 \right\}.\]
			Therefore, \(Q^r(R)=M_3(k)\) and \(Q^s(R)=R\). 
		\end{example}
		\begin{example}
			Let \(R= \begin{pmatrix} k& k & k\\ 0 &	k & 0\\ 0&0&k\end{pmatrix} \). 
			Note that \(Q_{max}^r(R)=M_2(k)\times M_2(k)\). The embedding of \(R\) into \(M_2(k)\times M_2(k)\) can be defined by
			\[\begin{pmatrix} a&b&c\\0&d&0 \\0 & 0 &e \end{pmatrix} \mapsto
			\left(\begin{pmatrix}a &b\\0&d \end{pmatrix}  , \begin{pmatrix}a &c\\0&e \end{pmatrix} \right).\]
			The collection of all non-zero ideals of \(R\) is
			\[\left\{
			I_1=\begin{pmatrix} 0& 0 & k\\ 0 &0& 0\\ 0&0&k\end{pmatrix} ,
			I_2=\begin{pmatrix} 0& k & 0\\ 0 &k& 0\\ 0&0&0\end{pmatrix} ,
			I_3=\begin{pmatrix} 0& k & k\\ 0 &k& 0\\ 0&0&k\end{pmatrix} ,
			I_4=\begin{pmatrix} k& k & k\\ 0 &k& 0\\ 0&0&0\end{pmatrix} ,
			I_5=R
			\right\}.\]
			The maximal right and the maximal left Martindale sets of \(R\) are respectively
			\[\overline T_r=\left\{I_3, I_5 \right\} \text{ and } \overline T_l=\left\{I_4, I_5 \right\}.\]
			Therefore, \(Q^r(R)=M_2(k)\times M_2(k)\) and \(Q^s(R)=R\). 
		\end{example}


\begin{thebibliography}{99}
	\bibitem{Amitsur72}
	Amitsur, S.A. (1972). \textit{On rings of quotients}, Symposia Math. \textbf{8}:149--164.
	
	\bibitem{Beidar96}
	 Beidar, K.I., Martindale, W.S. III, Mikhalev, A.V. (1996). Rings with Generalized Identities, Marcel Dekker, Inc., New York.
	
	\bibitem{Goodearl76}
	 Goodearl, K.R. (1976). Ring Theory: Nonsingular rings and modules, Marcel Dekker, Inc., New York.
	
	\bibitem{Lam91}
	Lam, T.Y. (1991). A First Course in Noncommutative Rings, Graduate Texts in Math., Vol.131, Springer-Verlag,
	Berlin-Heidelberg-New York.
	
	\bibitem{Lam99}
	 Lam, T.Y. (1999). Lectures on Modules and Rings, Graduate Texts in Math., Vol.189, Springer-Verlag,
	Berlin-Heidelberg-New York.
	
	\bibitem{Lambek66}
	 Lambek, J. (1966). Lectures on Rings and Modules, Blaisdell, Waltham, Mass.
	
	 \bibitem{Lanning96}
	 Lanning, S. (1996). \textit{The Maximal Symmetric Ring of Quotients},  Journal of Algebra \textbf{179}:47--91.
	
	\bibitem{Martindale69}
	Martindale, W.S. (1969). \textit{Prime rings satisfying a generalized polynomial identity}, Journal of Algebra \textbf{12}:576--584.
	
	\bibitem{Passman87}
	 Passman, D.S. (1987). \textit{Computing the Symmetric Ring of Quotients}, Journal of Algebra \textbf{105}:207--235.
	
	\bibitem{Passman89}
	 Passman, D.S. (1989). Infinite Crossed Products, Academic Press.
	
\bibitem{Passman91}
	 Passman, D.S. (1991). A Course in Ring Theory, Wadsworth \& Brooks/Cole, Pacific Grove, Calif.	
	

	
	\bibitem{Utumi56}
	Utumi, Y. (1956). \textit{On Quotient Rings}, Osaka Mathematical Journal \textbf{8}(1):1--18.
	
\end{thebibliography}

\end{document}